\documentclass[12pt, a4paper]{article}

\usepackage{graphicx}  
\usepackage[margin=3cm]{geometry}

\usepackage{authblk}

\usepackage{amsmath, amssymb}
\usepackage{bm}
\usepackage{amsthm}
\newtheorem{theorem}{Theorem}
\newtheorem{proposition}{Proposition}
\newtheorem{lemma}{Lemma}
\newtheorem{corollary}{Corollary}
\theoremstyle{definition}
\newtheorem{definition}{Definition}
\newtheorem{remark}{Remark}

\renewcommand{\theenumi}{\alph{enumi}}
\renewcommand{\labelenumi}{{\rm \theenumi)}}

\newcommand{\CC}{\mathbb{C}}
\newcommand{\HH}{\mathbb{H}}
\newcommand{\NN}{\mathbb{N}}
\newcommand{\RR}{\mathbb{R}}
\newcommand{\EE}{\mathop{\mathbb{E}}\nolimits}
\newcommand{\PP}{\mathop{\mathbb{P}}\nolimits}
\DeclareMathOperator{\dist}{dist}
\DeclareMathOperator{\hcap}{hcap}
\DeclareMathOperator{\Hm}{hm}
\DeclareMathOperator{\Lip}{Lip}
\DeclareMathOperator*{\Res}{Res}

\title{On the continuity of half-plane capacity \\ with respect to Carath\'eodory convergence}
\author{Takuya Murayama}
\affil{\small Department of Physics, Faculty of Science and Engineering, Chuo University, Tokyo 112-8551, Japan. \\
Research Fellow of Japan Society for the Promotion of Science. \\
E-mail: \texttt{murayama@phys.chuo-u.ac.jp}}
\date{}

\begin{document}

\maketitle


\begin{abstract}
We study the continuity of half-plane capacity as a function of boundary hulls with respect to the Carath\'eodory convergence.
In particular, our interest lies in the case that hulls are unbounded.
Under the assumption that every hull is contained in a fixed hull with finite imaginary part and finite half-plane capacity,
we show that the half-plane capacity is indeed continuous.
We also discuss the extension of this result to the case that the underlying domain is finitely connected.



\medskip \noindent
\emph{Acknowledgments}.
The author wishes to express his thanks to Professor Gregory Markowsky for the comments on the first manuscript, which helped the author to improve the exposition around uniform regularity and weak convergence of harmonic measures.
This research was supported by JSPS KAKENHI Grant Number JP21J00656.
\end{abstract}


\section{Introduction}
\label{sec:intro}

The \emph{half-plane capacity} is a ``capacity'' that measures \emph{$\HH$-hulls} (or \emph{boundary hulls}) growing from the boundary of the complex upper half-plane $\HH$.
This capacity has several geometric meanings~\cite{LLN09,RW14,DV14} and plays a role of time for chordal Loewner chains, that is, families of normalized conformal mappings defined in the complements of expanding (or shrinking) $\HH$-hulls.
In particular, the concept of Loewner chain is commonly utilized not just in complex analysis but also in probability theory due to the great importance of Schramm--Loewner evolution (SLE).
Roughly speaking, the SLE hull is a random curve in the half-plane with all of its loops filled-in.
Also, $\HH$-hulls of more general form appear naturally in some applications;
A recent example is the correspondence between monotone-independent increment processes and chordal Loewner chains established by Franz, Hasebe and Schlei{\ss}inger~\cite{FHS20} in non-commutative probability theory.
Under such a background, we study the continuity of half-plane capacity, as a function of $\HH$-hulls of broad class, with respect to the \emph{Carath\'eodory kernel convergence} in this paper.

To be precise, concepts studied in this paper are defined as follows:
A set $F \subset \HH$ is called an $\HH$-hull if $F$ is relatively closed in $\HH$ (i.e., $F = \overline{F} \cap \HH$) and if $\HH \setminus F$ is a simply connected domain.
Given an appropriate sequence of $\HH$-hulls,
we can define its convergence in Carath\'eodory's sense (see Definition~\ref{def:hull_continuity} in Section~\ref{sec:main1}).
Let $Z = ((Z_t)_{t \geq 0}, (\PP_z)_{z \in \HH})$
be an absorbed Brownian motion (ABM for brevity) in $\HH$.
The half-plane capacity of $F$ is defined by%
\footnote{
We adopt the usual convention that a function on $\HH$ takes value zero at the cemetery of the ABM.
The second equality in \eqref{eq:def} is understood in this sense.
}
\begin{equation}
\label{eq:def}
\hcap(F) := \lim_{y \to \infty} y \EE_{iy}\bigl[ \Im Z_{\sigma_F} \bigr] = \lim_{y \to \infty} y \EE_{iy} \bigl[ \Im Z_{\tau_{\HH \setminus F}} \bigr]
\end{equation}
as long as this limit exists.
Here, $\sigma_F$ (resp.\ $\tau_{\HH \setminus F}$) is the first hitting time (resp.\ exit time) of $Z$ to $F$ (resp.\ from $\HH \setminus F$).

Originally, the half-plane capacity appears in a purely analytic way.
Let $F$ be a \emph{bounded} $\HH$-hull.
Using Riemann's mapping theorem, we can show that
there exists a unique conformal mapping $g_F \colon \HH \setminus F \to \HH$ with  Laurent expansion
\begin{equation}
\label{eq:Laurent}
g_F(z) = z + \frac{a_F}{z} + o(z^{-1}),
\quad z \to \infty,
\end{equation}
around the point at infinity.
($g_F$ is sometimes called the mapping-out function of $F$.)
The non-negative constant $a_F$ is exactly the half-plane capacity of $F$.
We note that
\[
a_F = -\Res_{z=\infty} g_F(z) \, dz = \Res_{z=\infty} g_F^{-1}(z) \, dz.
\]

If hulls are assumed to be uniformly bounded, then the continuity of half-plane capacity is proved quite easily.
Let $F_n$, $n \in \NN \cup \{\infty\}$, be $\HH$-hulls that are contained in a disk $B(0, \rho)$ with center $0$ and radius $\rho$.
We suppose that $F_n$ converges to $F_\infty$ in Carath\'eodory's sense.
Using a version of the Carath\'eodory kernel theorem~\cite[Theorem~3.8]{Mu19spa} and taking the Schwarz reflection of $g_{F_n}$'s across the real axis,
we can show that $g_{F_n}(z)$ converges to $g_{F_\infty}(z)$ uniformly in $z \in \partial B(0, \rho)$.
By \eqref{eq:Laurent} we have
\[
\lim_{n \to \infty} a_{F_n}
= \lim_{n \to \infty} \frac{1}{2\pi i} \int_{\lvert z \rvert = \rho} g_{F_n}(z) \, dz
= \frac{1}{2\pi i} \int_{\lvert z \rvert = \rho} g_{F_\infty}(z) \, dz
= a_{F_\infty}.
\]

In contrast to the preceding case,
the continuity of half-plane capacity fails if hulls are allowed to be unbounded.
In this case, $\hcap(F)$ should be the \emph{angular} residue of $g_F^{-1}$ at infinity,
but the angular residue is not a continuous functional of $g_F^{-1}$
(see Goryainov and Ba~\cite[p.1211]{GB92}).
In order to retrieve the continuity of half-plane capacity,
some restriction must be imposed on $\HH$-hulls.

In this paper, we assume that our boundary hulls are contained in a fixed hull with finite imaginary part and finite half-plane capacity.
Under this assumption,
our main result is stated as follows:
\begin{theorem}
\label{thm:main1}
Let $\tilde{F}$ be an $\HH$-hull with
\[
\Im \tilde{F} := \sup \{\, \Im z \mathrel{;} z \in \tilde{F} \,\} < \infty \quad \text{and} \quad \hcap(\tilde{F})<\infty.
\]
Suppose that $\HH$-hulls $F_n$, $n \in \NN$, and $F_\infty$ are all contained in $\tilde{F}$
and that $F_n$ converges to $F_\infty$ as $n \to \infty$ (in the sense of Definition~\ref{def:hull_continuity}).
Then
\begin{equation}
\label{eq:main1}
\lim_{n \to \infty} \hcap(F_n) = \hcap(F_\infty).
\end{equation}
\end{theorem}

For the proof of Theorem~\ref{thm:main1},
we make a full use of the probabilistic definition~\eqref{eq:def} of half-plane capacity instead of (angular) residue.
The merit of our method is that the hitting probability of the ABM to a hull $F$ is given by the harmonic measure of $\HH \setminus F$.
In fact, the Carath\'edory convergence implies the weak convergence of harmonic measures under a certain assumption,
which is a key result to proving \eqref{eq:main1}.
Modifying the original argument of Binder, Rojas and Yampolsky~\cite{BRY19},
we provide a proof of this result in Appendix~\ref{sec:bry}.

After the proof of Theorem~\ref{thm:main1}, we discuss its extension to \emph{finitely connected domains}.
Let $D$ be a \emph{parallel slit half-plane}, namely, the upper half-plane with some line segments parallel to the real axis removed.
For an $\HH$-hull $F$ in $D$, its \emph{half-plane capacity} $\hcap^D(F)$ \emph{relative to} $D$ is defined by replacing ABM with \emph{Brownian motion with darning} (BMD for short) in the expression~\eqref{eq:def}.
Such replacement is naturally considered in the study of the Komatu--Loewner differential equation;
See a series of recent papers~\cite{CFR16,CF18,CFS17}.
We prove the continuity of $\hcap^D$ so defined under such assumptions as in Theorem~\ref{thm:main1} (see Theorem~\ref{thm:main2} in Section~\ref{sec:BMD_hcap}).

\begin{remark} \label{rem:jw}
The question of the continuity of half-plane capacity originally arose from the joint work~\cite{HHMS21+}, which is in progress.
In this work, a very similar result is proved by means of angular residues and an integral formula for conformal mappings.
(We shall mention the relation of Theorem~\ref{thm:main1} to this result again in Remark~\ref{rem:jw2} in Section~\ref{sec:rem2}.)
Thus, our main contribution lies in the simple probabilistic method, which works  for finitely connected domains as well, rather than the results themselves.
\end{remark}

The rest of this paper is organized as follows:
Section~\ref{sec:sc} is devoted to the case that the underlying domain is simply connected, namely, the upper half-plane.
In Section~\ref{sec:main1}, we provide the definition of Carath\'eodory's convergence and prove Theorem~\ref{thm:main1}.
In Section~\ref{sec:mono}, we show that $\hcap$ is strictly monotone (with respect to the inclusion relation of $\HH$-hulls)
and give a partial converse of Theorem~\ref{thm:main1} for a monotone sequence of hulls.
Section~\ref{sec:fc} is devoted to the case that the underlying domain is finitely connected, namely, a parallel slit half-plane.
In Section~\ref{sec:BMD_hcap}, we generalize the definition of half-plane capacity in terms of BMD.
The proof of Theorem~\ref{thm:main2} is done through Sections~\ref{sec:mc} and \ref{sec:ur}.
The final section, Section~\ref{sec:cc_rem}, is devoted to some remarks on the relation of our results to geometric function theory.
(The reader can read Section~\ref{sec:cc_rem} independently of Section~\ref{sec:fc}.)
There are two appendices.
In Appendix~\ref{sec:appendix}, we prove a lemma on some hitting probability needed in the proof of Proposition~\ref{prop:ur}.
In Appendix~\ref{sec:bry}, we prove the above-mentioned fact on the weak convergence of harmonic measures.

\section{Study on the upper half-plane}
\label{sec:sc}

In this section, we study boundary hulls in the upper half-plane $\HH$.
Our main result, Theorem~\ref{thm:main1}, is proved after some definitions are given.
We also discuss the case that a sequence of hulls is monotone.
Most of the results in this section will be carried over into the setting of Section~\ref{sec:fc}.

\subsection{Basic definitions and proof of Theorem~\ref{thm:main1}}
\label{sec:main1}

\begin{definition}[Convergence of domains]
\label{def:kernel}
Let $D_n$, $n \in \NN$, and $D_\infty$ be domains in the complex plane $\CC$ which have a point $z_0 \in \CC$ in common.
It is said that the sequence $(D_n)_{n \in \NN}$
\emph{converges} to $D_\infty$ \emph{in the kernel sense}
or \emph{in Carath\'eodory's sense}
with respect to the reference point $z_0$
if the following hold:
\begin{itemize}
\item
Each compact subset $K$ of $D_\infty$ is a subset of $D_n$ for all but finitely many $n$;

\item
If a domain $U$ contains $z_0$ and is a subset of $D_n$ for infinitely many $n$, then $U \subset D_\infty$.
\end{itemize}
\end{definition}

In Definition~\ref{def:kernel}, we have skipped the definition of ``kernel'' and defined the kernel convergence directly.
The equivalence of Definition~\ref{def:kernel} to the original definition (see for instance Roseblum and Rovnyak~\cite[\S7.9]{RR94}) is easy to check and left to the interested reader.

\begin{definition}[Convergence of $\mathbb{H}$-hulls]
\label{def:hull_continuity}
Suppose that $\HH$-hulls $F_n$, $n \in \NN$, and $F_\infty$ are contained in another hull $\tilde{F}$.
We say that the sequence $(F_n)_{n \in \NN}$ \emph{converges} to $F_\infty$ (in Carath\'eodory's sense)
if the complement $\HH \setminus F_n$ converges to $\HH \setminus F_\infty$ as $n \to \infty$ in the kernel sense with respect to some $z_0 \in \HH \setminus \tilde{F}$.
(This definition is independent of
the choice of the reference point $z_0$.)
\end{definition}

We use the following notation:
$Z = ((Z_t)_{t \geq 0}, (\PP_z)_{z \in \HH})$
is an ABM in $\HH$.
For a set $B \subset \HH$,
the symbol $\sigma_B$ (resp.\ $\tau_B$) denotes the first hitting time (resp.\ exit time) of $Z$ to $B$ (resp.\ from $B$).
For a domain $D$,
the expression
\[
\Hm_D(z, B) := \PP_z(Z_{\tau_D-} \in B)
\]
defines the \emph{harmonic measure} of $B$ in $D$ seen from a point $z$.
It is regarded as a Borel measure on $\CC$.
The value of any function at the cemetery of $Z$ is set to be zero.

Let us begin the proof of Theorem~\ref{thm:main1}.
The following observation comes from the expression of half-plane capacity mentioned by Lalley, Lawler and Narayanan~\cite{LLN09}:
For an $\HH$-hull $F$ with $\Im F<\infty$,
let $y > \eta > \Im F$ and $\HH_\eta := \{\, z \in \CC \mathrel{;} \Im z > \eta \,\}$.
Then
\begin{align*}
\EE_{iy} \left[ \Im Z_{\sigma_F} \right]
&= \EE_{iy} \left[ \EE_{Z_{\tau_{\HH_\eta}}} \! \left[ \Im Z_{\sigma_F} \right] \right]
= \int_{\partial \HH_\eta} \EE_\zeta \left[ \Im Z_{\sigma_F} \right] \Hm_{\HH_\eta}(iy, d\zeta) \\
&= \frac{1}{\pi} \int_\RR \EE_{\xi+i\eta} \left[ \Im Z_{\sigma_F} \right] \frac{y-\eta}{\xi^2 + (y-\eta)^2} \, d\xi.
\end{align*}
Hence we have
\begin{equation}
\label{eq:pre_LLN}
y \EE_{iy} \left[ \Im Z_{\sigma_F} \right]
= \frac{1}{\pi} \int_\RR \EE_{\xi+i\eta} \left[ \Im Z_{\sigma_F} \right] \frac{y(y-\eta)}{\xi^2 + (y-\eta)^2} \, d\xi
\quad \text{for}\ y > \eta > \Im F.
\end{equation}

\begin{proposition}[Expression of $\hcap$]
\label{prop:LLN}
Let $F$ be an $\HH$-hull with $\Im F<\infty$ and $\hcap(F)<\infty$.
Then, for any $\eta > \Im F$,
\begin{equation}
\label{eq:LLN}
\hcap(F)
=\frac{1}{\pi} \int_\RR \EE_{\xi+i\eta} \left[ \Im Z_{\sigma_F} \right] d\xi.
\end{equation}
\end{proposition}

\begin{proof}
We apply Fatou's lemma to \eqref{eq:pre_LLN}:
\[
\frac{1}{\pi} \int_\RR \EE_{\xi+i\eta} \left[ \Im Z_{\sigma_F} \right] d\xi
\leq \liminf_{y \to \infty} y \EE_{iy} \left[\Im Z_{\sigma_F}\right]
= \hcap(F) < \infty.
\]
Hence the function
$\xi \mapsto \EE_{\xi+i\eta} \left[ \Im Z_{\sigma_F} \right]$
is integrable. Moreover,
\begin{equation}
\label{eq:Poi_dom}
\frac{y(y-\eta)}{\xi^2 + (y-\eta)^2} \leq \frac{y}{y-\eta} < 2
\quad \text{for all}\ \xi \in \RR\ \text{and}\ y > 2\eta.
\end{equation}
The dominated convergence theorem thus applies to \eqref{eq:pre_LLN},
which yields \eqref{eq:LLN}.
\end{proof}

\begin{proposition}[Weak monotonicity]
\label{prop:mono}
Let $F$ and $\tilde{F}$ be $\HH$-hulls with $F \subset \tilde{F}$.
If $\tilde{F}$ satisfies $\Im \tilde{F}<\infty$ and $\hcap(\tilde{F})<\infty$,
then the limit $\hcap(F)$ exists and enjoys
\begin{equation}
\label{eq:mono}
\hcap(F) \leq \hcap(\tilde{F}).
\end{equation}
\end{proposition}

\begin{proof}
The process $(\Im Z_t)_{t \geq 0}$ is just a one-dimensional Brownian motion stopped when it hits the origin.
Hence it is a non-negative martingale.
Since $\lim_{t \to \infty} \Im Z_t = 0$ a.s., it is also a supermartingle with last element zero.
Thus, the optional sampling theorem implies that
\begin{equation}
\label{eq:ost}
\EE_{\xi+i\eta} \left[\Im Z_{\sigma_F}\right]
\leq \EE_{\xi+i\eta} \left[\Im Z_{\sigma_{\tilde{F}}}\right]
\quad \text{for}\ \xi \in \RR\ \text{and}\ \eta > \Im \tilde{F}.
\end{equation}
The right-hand side of \eqref{eq:ost} is integrable as a function of $\xi$ by Proposition~\ref{prop:LLN}.
Thus, by virtue of \eqref{eq:Poi_dom}, the dominated convergence theorem works in \eqref{eq:pre_LLN}.
This yields \eqref{eq:mono}.
\end{proof}

We now consider $\HH$-hulls $F_n$, $n \in \NN \cup \{\infty\}$, in Theorem~\ref{thm:main1}.
The next proposition then applies to the complementary domains $\HH \setminus F_n$.

\begin{proposition}[Uniform regularity of simply connected domains]
\label{prop:ur_sc}
Simply connected domains none of which is $\CC$ are uniformly regular.
Namely, for any $\varepsilon>0$ there exists $\delta>0$ such that every simply connected domain $D \subsetneq \CC$ satisfies
\[
\Hm_D(z, B(z,\varepsilon))>1-\varepsilon
\quad \text{for any}\ z \in D\ \text{with}\ \dist(z, \partial D) < \delta.
\]
Here, $B(z,\varepsilon)$ denotes the disk with center $z$ and radius $\varepsilon$, and $\dist(z, \partial D) := \inf\{\, \lvert z-w \rvert \mathrel{;} w \in \partial D\,\}$.
\end{proposition}

There are several proofs of Proposition~\ref{prop:ur_sc}, and we refer the reader to Markowsky~\cite[Lemma~1]{Ma18} for a short and elegant probabilistic one.
Using this proposition, we prove the following:

\begin{proposition}
\label{prop:EE_conv}
Under the assumption of Theorem~\ref{thm:main1},
it holds that
\begin{equation}
\label{eq:EE_conv}
\lim_{n \to \infty} \EE_\zeta \left[ \Im Z_{\sigma_{F_n}} \right]
= \EE_\zeta \left[ \Im Z_{\sigma_{F_\infty}} \right]
\quad \text{for every}\ \zeta \in \HH \setminus \tilde{F}.
\end{equation}
\end{proposition}

\begin{proof}
The domains $\HH \setminus F_n$, $n \in \NN \cup \{\infty\}$, are uniformly regular by Proposition~\ref{prop:ur_sc},
and $\HH \setminus F_n$ converges to $\HH \setminus F_\infty$ in Carath\'eodory's sense by assumption.
Thus, it follows from Theorem~\ref{thm:BRY19} in Appendix~\ref{sec:bry} that
$\Hm_{\HH \setminus F_n}(\zeta, \cdot)$
converges weakly to
$\Hm_{\HH \setminus F_\infty}(\zeta, \cdot)$.
Putting $\psi(z) := \max \{0, \min\{\Im z, \Im \tilde{F}\}\}$,
we have
\begin{align*}
\lim_{n \to \infty} \EE_\zeta \left[ \Im Z_{\sigma_{F_n}} \right]
&= \lim_{n \to \infty} \int_\CC \psi(z) \Hm_{\HH \setminus F_n}(\zeta, dz) \\
&= \int_\CC \psi(z) \Hm_{\HH \setminus F_\infty}(\zeta, dz)
= \EE_\zeta \left[ \Im Z_{\sigma_{F_\infty}} \right].
\qedhere
\end{align*}
\end{proof}

Proposition~\ref{prop:EE_conv} completes the proof of Theorem~\ref{thm:main1}.
Indeed, Proposition~\ref{prop:mono} ensures that $\hcap(F_n) < \infty$ for all $n$.
Then by virtue of \eqref{eq:EE_conv} and \eqref{eq:ost} with $F=F_n$,
the dominated convergence theorem applies to \eqref{eq:LLN} with $F = F_n$.
This yields \eqref{eq:main1}.

\subsection{Strict monotonicity}
\label{sec:mono}

In Proposition~\ref{prop:mono}, we have seen that the half-plane capacity is a weakly monotone function of $\HH$-hulls.
In fact, it is strictly monotone as follows:

\begin{proposition}[Strict monotonicity]
\label{prop:str_mono}
Let $F$ and $\tilde{F}$ be $\HH$-hulls with $F \subsetneq \tilde{F}$, $\Im \tilde{F} < \infty$, and $\hcap(\tilde{F}) < \infty$.
Then
\[
\hcap(F) < \hcap(\tilde{F}).
\]
\end{proposition}

\begin{proof}
We assume
\begin{equation} \label{eq:not_mono}
\hcap(F) = \hcap(\tilde{F})
\end{equation}
to the contrary and deduce a contradiction.

By \eqref{eq:LLN}, \eqref{eq:ost}, and \eqref{eq:not_mono}, we have, for each fixed $\eta > \Im \tilde{F}$,
\[
\EE_{\xi+i\eta}\left[\Im Z_{\sigma_F}\right]
= \EE_{\xi+i\eta}\left[\Im Z_{\sigma_{\tilde{F}}}\right]
\quad \text{for a.e.}\ \xi \in \RR.
\]
Both sides of this equality are harmonic functions of variable $\zeta = \xi+i\eta$.
Hence the identity theorem applies:
\[
\EE_\zeta\left[\Im Z_{\sigma_F}\right]
= \EE_\zeta\left[\Im Z_{\sigma_{\tilde{F}}}\right]
\quad \text{for any}\ \zeta \in \HH \setminus \tilde{F}.
\]
In particular, if $z_0 \in (\HH \cap \partial \tilde{F}) \setminus F$, then
\begin{equation} \label{eq:max_u}
\EE_{z_0}\left[\Im Z_{\sigma_F}\right]
= \EE_{z_0}\left[\Im Z_{\sigma_{\tilde{F}}}\right] = \Im z_0
\end{equation}
because $z_0$ is a regular point of $\partial (\HH \setminus \tilde{F})$ by Proposition~\ref{prop:ur_sc}.
Here, the set $(\HH \cap \partial \tilde{F}) \setminus F$ is not empty.
Otherwise, $\tilde{F} \setminus F = \tilde{F}^\circ \setminus F$ would be open ($\tilde{F}^\circ$ stands for the interior of $\tilde{F}$),
and the domain $\HH \setminus F$ would be divided into disjoint open sets $\HH \setminus \tilde{F}$ and $\tilde{F} \setminus F$.

Now, we define a harmonic function
\[
u(z) := \EE_z\left[\Im Z_{\sigma_F}\right] - \Im z,
\quad z \in \HH \setminus F.
\]
Since $u(z)$ enjoys
\[
\lim_{z \to \zeta}u(z) = 0\ \text{for any}\ \zeta \in \partial(\HH \setminus F)
\quad \text{and} \quad
\limsup_{z \to \infty} \frac{u(z)}{\log\lvert z \rvert} \le 0,
\]
a corollary of the Phragm\'en--Lindel\"of principle~\cite[Corollary~2.3.3]{Ra95} yields $u \le 0$ in $\HH \setminus F$.
Hence \eqref{eq:max_u} implies that $u$ takes its maximum at $z_0$.
By the maximum principle, we have $u \equiv 0$, namely,
\[
\EE_z\left[\Im Z_{\sigma_F}\right] = \Im z
\quad \text{for all}\ z \in \HH \setminus F.
\]
This is absurd.
\end{proof}

Using the strict monotonicity, we can show that $\hcap(F_n) \to \hcap(F_\infty)$ implies $F_n \to F_\infty$,
the converse of Theorem~\ref{thm:main1},
if $(F_n)_{n \in \NN}$ is monotone.
To this end, we recall the limit of monotone hulls.

\begin{proposition}[Limit of monotone hulls]
\label{prop:lim_mono_hull}
Let $(F_n)_{n \in \NN}$ be a sequence of $\HH$-hulls.
\begin{enumerate}
\item \label{item:declim}
Suppose that $(F_n)_{n \in \NN}$ is decreasing and set $F := \bigcap_{n \in \NN}F_n$.
Then $F$ is an $\HH$-hull, and $F_n \to F$ in Carath\'eodory's sense (with respect to any point $z_0 \in \HH \setminus F_0$).

\item \label{item:inclim}
Suppose that $(F_n)_{n \in \NN}$ is increasing with $G := \bigcup_{n \in \NN}F_n \subset \tilde{F}$ for some hull $\tilde{F}$.
Let $F$ be the union of $\overline{G} \cap \HH$ and all the connected components of $\HH \setminus \overline{G}$ that are disjoint from $\HH \setminus \tilde{F}$.
Then $F$ is an $\HH$-hull, and $F_n \to F$ in Carath\'eodory's sense (with respect to any point $z_0 \in \HH \setminus \tilde{F}$).
\end{enumerate}
\end{proposition}

Proposition~\ref{prop:lim_mono_hull} is well known if $\HH$-hulls are uniformly bounded.
It is also not so difficult to prove on the basis of Definitions~\ref{def:kernel} and \ref{def:hull_continuity}.
We omit the detail here.

\begin{proposition}
\label{prop:converse_main}
Let $F_n$, $n \in \NN \cup \{+\infty\}$, be a monotone sequence of $\HH$-hulls.
Suppose that there exists a hull $\tilde{F}$ such that
\[
\bigcup_n F_n \subset \tilde{F},\ \Im \tilde{F}<\infty,\ \text{and}\ \hcap(\tilde{F}) < \infty.
\]
Then $\lim_{n \to \infty}\hcap(F_n) = \hcap(F_\infty)$ implies $F_n \to F_\infty$ in Carath\'eodory's sense.
\end{proposition}

\begin{proof}
We assume first that $F_n$, $n \in \NN \cup \{+\infty\}$, are decreasing.
Let $F := \bigcap_{n \in \NN}F_n \supset F_\infty$.
We have $F_n \to F$ by Proposition~\ref{prop:lim_mono_hull}~\eqref{item:declim}
and hence $\hcap(F_n) \to \hcap(F)$ by Theorem~\ref{thm:main1}.
Thus, we obtain
\[
F_\infty \subset F
\quad \text{and} \quad
\hcap(F_\infty) = \hcap(F).
\]
Proposition~\ref{prop:str_mono} yields $F_\infty = F$.

We assume next that $F_n$, $n \in \NN \cup \{+\infty\}$, are increasing.
Let $F$ be the hull defined in Proposition~\ref{prop:lim_mono_hull}~\eqref{item:inclim}.
We have $F_n \to F$ and hence $\hcap(F_n) \to \hcap(F)$ by Theorem~\ref{thm:main1}.
Since $\HH \setminus F_\infty \subset \bigcap_{n \in \NN}(\HH \setminus F_n)$,
Definition~\ref{def:kernel} implies that $\HH \setminus F_\infty \subset \HH \setminus F$.
Thus,
\[
F_\infty \supset F
\quad \text{and} \quad
\hcap(F_\infty) = \hcap(F).
\]
Proposition~\ref{prop:str_mono} yields $F_\infty = F$ again.
\end{proof}

\section{Study on parallel slit half-planes}
\label{sec:fc}

In this section, we formulate and prove an extension of Theorem~\ref{thm:main1} in parallel slit half-planes,
a standard type of finitely connected domains.

\subsection{BMD half-plane capacity}
\label{sec:BMD_hcap}

Let $N \in \NN \setminus \{0\}$.
For disjoint horizontal line segments $C_j$, $j=1, 2, \ldots, N$, in $\HH$,
we put $K := \bigcup_{j=1}^N C_j$ and $D := \HH \setminus K$.
Such a domain $D$ is called a parallel slit half-plane.
We also consider the quotient space $D^\ast$ of $\HH$
in which each $C_j$ is identified with one point $c^\ast_j$.
With $K^\ast := \{ c^\ast_1, c^\ast_2, \ldots, c^\ast_N \}$,
it is written as $D^\ast := D \cup K^\ast$.

A BMD in $D^\ast$, which we denote by $Z^\ast = ((Z^\ast_t)_{t \geq 0}, (\PP^\ast_z)_{z \in D^\ast})$, is a symmetric%
\footnote{
`Symmetric' means ``symmetric with respect to the Lebesgue measure on $D^\ast$.''
The measure of $K^\ast$ is set to be zero.
}
diffusion process in $D^\ast$ with the following two properties:
\begin{itemize}
\item The killed process of $Z^\ast$ when it exits from $D$ is an ABM in $D$;
\item $Z^\ast$ admits no killing in $K^\ast$.
\end{itemize}
See Chen, Fukushima and Rohde~\cite{CFR16} for basic properties of BMD.

The (BMD) half-plane capacity of an $\HH$-hull $F \subset D$ relative to $D$
can be defined by
\begin{equation}
\label{eq:def2}
\hcap^D(F) := \lim_{y \to \infty} y \EE^\ast_{iy} \left[ \Im Z^\ast_{\sigma_F} \right]
\end{equation}
as long as this limit exists.
Here, $\sigma_F$ denotes the first hitting time of $Z^\ast$ to $F$.
As before, the right-hand side of \eqref{eq:def2} typically coincides with the (angular) residue of some conformal mapping from or onto $D \setminus F$ (see Chen and Fukushima~\cite[p.591]{CF18}),
which ensures the validity of this definition.

The main result of this section is the following:

\begin{theorem}
\label{thm:main2}
Let $\tilde{F} \subset D$ be an $\HH$-hull
with $\Im \tilde{F}<\infty$ and $\hcap^D(\tilde{F})<\infty$.
Suppose that $\HH$-hulls $F_n$, $n \in \NN$, and $F_\infty$ are all contained in $\tilde{F}$
and that $F_n$ converges to $F_\infty$ as $n \to \infty$
(in the sense of Definition~\ref{def:hull_continuity}).
Then
\begin{equation}
\label{eq:main2}
\lim_{n \to \infty} \hcap^D(F_n) = \hcap^D(F_\infty).
\end{equation}
\end{theorem}

Notice that $D \setminus F_n$ converges to $D \setminus F_\infty$ in Carath\'eodory's sense if $F_n \to F_\infty$ in the sense of Definition~\ref{def:hull_continuity}.
This is easily seen in view of Definition~\ref{def:kernel}.

Propositions~\ref{prop:LLN} and \ref{prop:mono} hold for BMD and BMD half-plane capacity with obvious modifications.
Thus, in order to prove Theorem~\ref{thm:main2}, it suffices to show the following:

\begin{proposition}
\label{prop:EE_conv2}
Under the assumption of Theorem~\ref{thm:main2},
it holds that
\begin{equation}
\label{eq:EE_conv2}
\lim_{n \to \infty} \EE^\ast_\zeta \left[ \Im Z^\ast_{\sigma_{F_n}} \right]
= \EE^\ast_\zeta \left[ \Im Z^\ast_{\sigma_{F_\infty}} \right]
\quad \text{for every}\ \zeta \in D \setminus \tilde{F}.
\end{equation}
\end{proposition}

We take two steps to prove Proposition~\ref{prop:EE_conv2}:
\begin{enumerate}
\item \label{item:idea1}
Express $\EE^\ast_\zeta \left[ \Im Z^\ast_{\sigma_{F_n}} \right]$ in terms of ABM;

\item \label{item:idea2}
Prove the uniform regularity (Definition~\ref{def:ur}) of $D \setminus F_n$'s to show that $\Hm_{D \setminus F_n}(\zeta, \cdot) \to \Hm_{D \setminus F_\infty}(\zeta, \cdot)$ weakly.
\end{enumerate}
To carry out the step~\ref{item:idea1}),
we study Markov chains on $K^\ast$ associated with BMD  in Section~\ref{sec:mc}.
To implement the step~\ref{item:idea2}), we make a comparison of $\Hm_{D \setminus F_n}$ and $\Hm_{\HH \setminus F_n}$, the latter of which behaves well by Proposition~\ref{prop:ur_sc}, in Section~\ref{sec:ur}.

\subsection{Markov chains induced by BMD}
\label{sec:mc}

Let $F$ and $\tilde{F}$ be $\HH$-hulls with $F \subset \tilde{F} \subset D$ and $\Im \tilde{F}<\infty$.
Following Lawler~\cite[Sections~5.2 and 5.3]{La06} and Chen, Fukushima and Rohde~\cite[Appendix~1]{CFR16},
we observe that BMD naturally induces Markov chains in $K^\ast \cup \{c^\ast_0\}$ ($c^\ast_0$ represents the cemetery of the chains).
We shall use their transition probabilities to compute the expectation with respect to BMD.

For each $j=1, 2, \ldots, N$, let $\eta_j$ be a smooth Jordan curve
in $D \setminus \tilde{F}$ surrounding $C_j$.
We define a measure on $\eta_j$ by
\[
\nu_j(dz) := \PP^\ast_{c^\ast_j} \left(Z^\ast_{\sigma_{\eta_j}} \in dz\right).
\]
We also put
\[
\varphi^{(j)}(z)
:= \Hm_{D \setminus F}(z, C_j)
= \PP^\ast_z(Z^\ast_{\sigma_{K^\ast}} = c^\ast_j, \sigma_{K^\ast} < \sigma_F)
\]
for $z \in D \setminus F$.
We consider a Markov chain $X = (X_n)_{n \in \NN}$
whose transition probability from $c^\ast_j$ to $c^\ast_k$,
$j, k \in \{1, 2, \ldots, N\}$, is given by
\begin{equation}
\label{eq:tp1}
p_{jk}
:= \EE^\ast_{c^\ast_j} \left[
	\PP^\ast_{Z^\ast_{\sigma_{\eta_j}}} \left(
		Z^\ast_{\sigma_{K^\ast}} = c^\ast_k, \sigma_{K^\ast} < \sigma_F
	\right)
	\right]
= \int_{\eta_j} \varphi^{(k)}(z) \, \nu_j(dz).
\end{equation}
Hence the chain $X$ moves from $c^\ast_j$ to $c^\ast_k$
when a BMD restricted in $D^\ast \setminus F$ moves from $c^\ast_j$ to $c^\ast_k$ after passing $\eta_j$.
The probabilities $p_{j0}$, $0 \leq j \leq N$, are defined in an obvious way.

We condition the chain $X$ defined above not to stay at the same state in one step.
Then the corresponding transition probability is given by
\[
q_{jk} = \frac{p_{jk}}{1 - p_{jj}}\ (j \neq 0),
\quad q_{0k} = \delta_{0k}.
\]
This conditioned chain satisfies
$q_{j0} > 0$ for all $j$.
Thus, the matrix
$Q := (q_{jk})_{j,k=1}^N$
has eigenvalues which are all less than one,
and the inverse $M = (M_{jk})_{j,k=1}^N := (I - Q)^{-1}$ exists.

Using the transition probabilities introduced above,
let us compute
\[
V^\ast(z) := \EE^\ast_{z}\left[ \Im Z^\ast_{\sigma_F} \right],
\quad z \in D^\ast \setminus F.
\]
We also define%
\footnote{
The symbol $\sigma_{\cdot}$ is used for two meanings here: the hitting times of the ABM $Z$ in $\HH$ and of the BMD $Z^\ast$ in $D^\ast$.
Although this is abuse of notation, there will be no confusion.
}, for the ABM $Z$ in $\HH$,
\[
V(z) := \EE_{z}\left[ \Im Z_{\sigma_F}; \sigma_F < \sigma_K \right],
\quad z \in D \setminus F.
\]

\begin{proposition}
\label{prop:BMD_to_ABM}
The function $V^\ast(z)$ satisfies
\begin{equation}
\label{eq:V_ast}
V^\ast(z) = V(z) + \sum_{j=1}^N \varphi^{(j)}(z) \sum_{k=1}^N \frac{M_{jk}}{1 - p_{kk}} \int_{\eta_k} V(z) \, \nu_k(dz),
\quad z \in D \setminus F.
\end{equation}
\end{proposition}

\begin{proof}
For $z \in D \setminus F$,
\begin{align}
V^\ast(z)
&= \EE^\ast_{z}\left[ \Im Z^\ast_{\sigma_F}; \sigma_F < \sigma_{K^\ast} \right]
	+ \EE^\ast_{z}\left[ \Im Z^\ast_{\sigma_F}; \sigma_{K^\ast} < \sigma_F \right] \notag \\
&= V(z) + \sum_{j=1}^N V^\ast(c^\ast_j) \PP^\ast_{z}\left( Z^\ast_{\sigma_{K^\ast}} = c^\ast_j, \sigma_{K^\ast} < \sigma_F \right) \notag \\
&= V(z) + \sum_{j=1}^N \varphi^{(j)}(z) V^\ast(c^\ast_j).
\label{eq:V_z}
\end{align}
Integrating the both side by $\nu_k$ and using the strong Markov property,
we have
\[
V(c^\ast_k)
= \int_{\eta_k} V(z) \, \nu_k(dz) + \sum_{j=1}^N p_{kj} V^\ast(c^\ast_j).
\]
This is equivalent to
\[
\sum_{j=1}^N (\delta_{kj} - q_{kj}) V(c^\ast_j) = \frac{1}{1 - p_{kk}} \int_{\eta_k} V(z) \, \nu_k(dz).
\]
Hence we finally get
\begin{equation}
\label{eq:V_c}
V^\ast(c^\ast_j) = \sum_{k=1}^N \frac{M_{jk}}{1 - p_{kk}} \int_{\eta_k} V(z) \, \nu_k(dz).
\end{equation}
Substituting \eqref{eq:V_c} into \eqref{eq:V_z} yields \eqref{eq:V_ast}.
\end{proof}

We now consider the case that $F$ coincides with $F_n$ in Theorem~\ref{thm:main2}.
In this case, we write the above functions $V^\ast$, $V$, and $\varphi^{(j)}$ as $V^\ast_n$, $V_n$, and $\varphi_n^{(j)}$, respectively.
We also denote the above $p_{jk}$ and $M_{jk}$ by $p^n_{jk}$ and $M^n_{jk}$, respectively.

\begin{proposition}
\label{prop:conv3}
If
\begin{equation}
\label{eq:Hm_conv}
\Hm_{D \setminus F_n}(\zeta, \cdot) \xrightarrow{w} \Hm_{D \setminus F_\infty}(\zeta, \cdot)
\ \text{as}\ n \to \infty\ \text{for every}\ \zeta \in D \setminus \tilde{F},
\end{equation}
then $V_n(\zeta)$, $\varphi_n^{(j)}(\zeta)$, $p^n_{jk}$, and $M^n_{jk}$ converge to $V_\infty(\zeta)$, $\varphi_\infty^{(j)}(\zeta)$, $p^\infty_{jk}$, and $M^\infty_{jk}$, respectively, as $n \to \infty$.
\end{proposition}

\begin{proof}
Suppose \eqref{eq:Hm_conv}.
Let $\psi_1(z)$ be a bounded continuous function which is equal to $\Im z$ on $\tilde{F}$ and zero on $K$.
Then
\begin{align*}
\lim_{n \to \infty} V_n(\zeta)
&= \lim_{n \to \infty} \int_\CC \psi_1(z) \Hm_{D \setminus F_n}(\zeta, dz) \\
&= \int_\CC \psi_1(z) \Hm_{D \setminus F_\infty}(\zeta, dz)
= V_\infty(\zeta).
\end{align*}
In order to show $\lim_{n \to \infty} \varphi^{(j)}_n(\zeta) = \varphi^{(j)}_\infty(\zeta)$, we just replace $\psi_1(z)$ by a bounded continuous function $\psi_2(z)$ which takes value one on $C_j$ and zero on $\tilde{F} \cup (K \setminus C_j)$.
Since $0 \leq \varphi^{(j)}_n(\zeta) \leq 1$, the dominated convergence theorem applies to \eqref{eq:tp1}, which yields $p^n_{jk} \to p^\infty_{jk}$.
Finally, $M^n_{jk} \to M^\infty_{jk}$ follows from the above-mentioned construction.
\end{proof}

From Propositions~\ref{prop:BMD_to_ABM} and \ref{prop:conv3}, we get the following:

\begin{corollary}
\label{cor:Hm_to_EE}
The convergence of harmonic measures \eqref{eq:Hm_conv} implies that of expectations \eqref{eq:EE_conv2} in Proposition~\ref{prop:EE_conv2}.
\end{corollary}

\subsection{Uniform regularity of slit domains}
\label{sec:ur}

In order to verify the convergence of harmonic measures \eqref{eq:Hm_conv},
we prove the uniform regularity of domains $D \setminus F_n$.

\begin{proposition}
\label{prop:ur}
Let $D$ and $F_n$, $n \in \NN \cup \{\infty\}$, be as in Theorem~\ref{thm:main2}.
Then the domains $D \setminus F_n$ are uniformly regular (in the sense of Definition~\ref{def:ur} in Appendix~\ref{sec:bry}).
\end{proposition}

\begin{proof}
In this proof, we consider
\[
\tilde{A} := \tilde{F} \cup \partial \HH
\quad \text{and} \quad
A_n := F_n \cup \partial \HH,\ n \in \NN \cup \{\infty\},
\]
instead of $\tilde{F}$ and $F_n$.
Needless to say, $D \setminus F_n = D \setminus A_n$.

Since $K$ and $\tilde{A}$ are disjoint, we have $r := \dist(K, \tilde{A}) > 0$.
If $\varepsilon \in (0, r/4)$, then $B(z, \varepsilon)$ intersects only one of the sets $K$ and $A_n$.
Therefore, we can divide the proof of the uniform regularity into two cases:
\begin{enumerate}
\item \label{item:case1}
$B(z, \varepsilon) \cap A_n \neq \emptyset$ and $B(z, \varepsilon) \cap K = \emptyset$;
\item \label{item:case2}
$B(z, \varepsilon) \cap A_n = \emptyset$ and $B(z, \varepsilon) \cap K \neq \emptyset$.
\end{enumerate}
In what follows, let $Z = ((Z_t)_{t \geq 0}, (\PP_z)_{z \in \CC})$ be a complex Brownian motion.

\smallskip \noindent
\ref{item:case1})
We consider the case~\ref{item:case1}).
Since the domains $\HH \setminus A_n$ are uniformly regular by Proposition~\ref{prop:ur_sc},
there exists $\delta \in (0, \varepsilon)$ such that
\begin{equation}
\label{eq:ur_HH_A}
1 - \varepsilon
< \Hm_{\HH \setminus A_n}(z, B(z, \varepsilon))
= \PP_z\left(Z_{\sigma_{A_n}} \in B(z, \varepsilon)\right)
\end{equation}
for $z \in D \setminus A_n$ with $\dist(z, A_n) < \delta$.
We decompose the right-hand side as follows:
\begin{align}
&\PP_z\left(Z_{\sigma_{A_n}} \in B(z, \varepsilon)\right) \notag \\
&= \PP_z\left(Z_{\sigma_{A_n}} \in B(z, \varepsilon), \sigma_{A_n} < \sigma_K\right) + \PP_z\left(Z_{\sigma_{A_n}} \in B(z, \varepsilon), \sigma_K < \sigma_{A_n}\right) \notag \\
&= \Hm_{D \setminus A_n}(z, B(z, \varepsilon)) + \EE_z\left[\PP_{Z_{\sigma_K}}\left(Z_{\sigma_{A_n}} \in B(z, \varepsilon)\right); \sigma_K < \sigma_{A_n}\right] \notag \\
&\leq \Hm_{D \setminus A_n}(z, B(z, \varepsilon)) + \sup_{w \in K} \PP_w\left(Z_{\sigma_{A_n}} \in B(z, \varepsilon) \right) \notag \\
&\leq \Hm_{D \setminus A_n}(z, B(z, \varepsilon)) + \sup_{w \in K} \PP_w\left(\sigma_{B(z, \varepsilon)}<\tau_{\overline{\HH}} \right).
\label{eq:st_decomp}
\end{align}
In the last expression, the following uniform convergence is not difficult to see:
\begin{equation}
\label{eq:radius0}
\lim_{\varepsilon \to 0} \sup \left\{\, \PP_w\left(\sigma_{B(z, \varepsilon)}<\tau_{\overline{\HH}} \right) \mathrel{;} z \in \Bar{N}_\delta(\tilde{A}),\ w \in K \,\right\} = 0.
\end{equation}
Here, a closed subset
\[
\Bar{N}_\delta(\tilde{A}) := \{\, z \in \overline{\HH} \mathrel{;} \dist(z, \tilde{A}) \leq \delta \,\}
\]
of $\overline{\HH}$ is disjoint from $K$ by definition.
We provide the proof of \eqref{eq:radius0} for the sake of completeness in Appendix~\ref{sec:appendix}.
Finally, \eqref{eq:ur_HH_A}, \eqref{eq:st_decomp}, and \eqref{eq:radius0} ensure the condition for the uniform regularity in the case~\ref{item:case1}).

\smallskip \noindent
\ref{item:case2})
We consider the case~\ref{item:case2}).
In this case, we have
\begin{align*}
\Hm_{D \setminus A_n}(z, B(z, \varepsilon))
&= \PP_z \left(Z_{\sigma_K} \in B(z, \varepsilon); \sigma_K < \sigma_{A_n}\right) \\
&\geq \PP_z \left(Z_{\sigma_K} \in B(z, \varepsilon); \sigma_K < \sigma_{\tilde{A}}\right) \\
&\geq \PP_z \left( \sigma_K \leq \tau_{B(z, \varepsilon)} \right)
= \Hm_{B(z, \varepsilon) \setminus K}(z, K).
\end{align*}
Hence it suffices to show that,
for some fixed $\varepsilon_0 \in (0, r/4]$ and
for every $\varepsilon \in (0, \varepsilon_0)$,
\begin{equation}
\label{eq:inner_bdry}
\lim_{\delta \to 0} \inf \left\{\, \Hm_{B(z, \varepsilon) \setminus K}(z, K) \mathrel{;} z \in \CC,\ \dist(z, K) < \delta \,\right\} = 1.
\end{equation}
There may be several ways to prove this, and one way is as follows:
If $\varepsilon_0$ is small enough, then for every $\varepsilon \in (0, \varepsilon_0)$, the disk $\overline{B(z, \varepsilon)}$ intersects only one slit $C_j$ of $K$,
and the length of $C_j$ is greater than $2\varepsilon$.
For such an $\varepsilon$, let $\delta \in (0, \varepsilon)$.
For any point $z$ with $\rho_z := \dist(z, K) < \delta$,
we put $E := C_j \cap \overline{B(z, \varepsilon)}$.
In the disk $\overline{B(z, \varepsilon)}$,
we consider the circular projection
\[
\hat{E} := \{\, \lvert w - z \rvert + z \mathrel{;} w \in E \,\}
\]
of $E$ onto the horizontal radius.
Clearly, $\hat{E}$ is the line segment that connects $z + \rho_z$ and $z + \varepsilon$.
By Beurling's projection theorem (see, e.g., Garnett and Marshall~\cite[Theorem~9.2, Chapter~III]{GM05}),
we have
\begin{align*}
\Hm_{B(z, \varepsilon) \setminus K}(z, K)
&\geq \Hm_{B(z, \varepsilon) \setminus \hat{E}}(z, \hat{E})
= \frac{2}{\pi} \arctan \left[ \frac{1}{2}\left(\sqrt{\frac{\varepsilon}{\rho_z}} - \sqrt{\frac{\rho_z}{\varepsilon}}\right) \right] \\
&\geq \frac{2}{\pi} \arctan \left[ \frac{1}{2}\left(\sqrt{\frac{\varepsilon}{\delta}} - \sqrt{\frac{\delta}{\varepsilon}}\right) \right].
\end{align*}
The last expression goes to one as $\delta \to 0$.
This proves \eqref{eq:inner_bdry}.
\end{proof}

Proposition~\ref{prop:ur} and Theorem~\ref{thm:BRY19} immediately yield the next corollary.

\begin{corollary}
\label{cor:ur}
Let $D$ and $F_n$, $n \in \NN \cup \{\infty\}$, be as in Theorem~\ref{thm:main2}.
Then \eqref{eq:Hm_conv} holds, that is, $\Hm_{D \setminus F_n}(\zeta, \cdot)$ converges weakly to $\Hm_{D \setminus F_\infty}(\zeta, \cdot)$ for every $\zeta \in D \setminus \tilde{F}$.
\end{corollary}

Corollaries~\ref{cor:Hm_to_EE} and \ref{cor:ur} prove Proposition~\ref{prop:EE_conv2} and hence Theorem~\ref{thm:main2}.

\section{Relation to geometric function theory}
\label{sec:cc_rem}

We give some remarks on Theorem~\ref{thm:main1} in the case that $\HH$-hulls are unbounded.
They are also the case with Theorem~\ref{thm:main2}.

\subsection{Half-plane capacity and angular residue at infinity}
\label{sec:rem1}

In view of geometric function theory, a natural way to define the half-plane capacity of an unbounded $\HH$-hull is to define it as the angular residue of the Riemann map at infinity.
More precisely, suppose that there exists a (unique) conformal mapping $f_F \colon \HH \to \HH \setminus F$ with the following two properties~\cite[Lemma~1~b)]{GB92}:
\begin{equation} \label{eq:hyd}
\lim_{\substack{z \to \infty \\ \Im z > \eta}}(f_F(z) - z) = 0
\quad \text{for any}\ \eta>0,
\end{equation}
and there exists $a_F \in \CC$ such that
\begin{equation} \label{eq:res}
\lim_{\substack{z \to \infty \\ \arg z \in (\theta, \pi-\theta)}}z(z - f_F(z)) = a_F
\quad \text{for any}\ \theta \in (0, \pi/2).
\end{equation}
The constant $a_F$ turns out to be non-negative and is called the angular residue at infinity.

It is known that the inverse mapping $f_F^{-1}$ has angular residue $-a_F$ at infinity.
Expressing the harmonic function $\Im (f_F^{-1}(z) - z)$ in terms of ABM, we can see that
\begin{itemize}
\item[A)] if  there exists a conformal mapping $f_F \colon \HH \to \HH \setminus F$ which enjoys \eqref{eq:hyd} and \eqref{eq:res}, then we have $\Im F < \infty$ and $\hcap(F) = a_F$.
\end{itemize}
In particular, the existence of the vertical limit~\eqref{eq:def} follows from that of the angular limit~\eqref{eq:res}.
On the other hand, it seems difficult to tell whether the following statement, which is the converse of A), is true or not:
\begin{itemize}
\item[B)] If the limit~\eqref{eq:def} exists, then there exists a conformal mapping $f_F \colon \HH \to \HH \setminus F$ which enjoys \eqref{eq:hyd} and \eqref{eq:res}.
\end{itemize}

In order to explain the difficulty of the above problem, suppose that the limit \eqref{eq:def} exists.
We can then take a conformal mapping $f \colon \HH \to \HH \setminus F$ with $f(\infty) = \infty$ in the sense of angular limit.
Compared with \eqref{eq:hyd} and \eqref{eq:res}, however, the last condition is too weak to relate the behavior of the inverse $f^{-1}(z)$ around $z=\infty$ to the quantity $\hcap(F)$.
Typical tools concerning holomorphic self-mappings in $\HH$ might be useful, such as the Pick--Nevanlinna integral representation, the Julia--Wolff--Carath\'eodory theorem, and so on,
but at this moment they do not directly imply that $f^{-1}(z)$ behaves well near $\infty$.

We note that some probabilistic methods could be available for constructing the conformal mapping $f_F$ above.
In the case that the hull $F$ is bounded, such a probabilistic construction of conformal mappings is studied by Lawler~\cite[Section~5.2]{La06} and by Chen, Fukushima and Rohde~\cite[Theorem~7.2]{CFR16}.
It will be a natural question whether we can obtain such results in the present case as well.

\subsection{Carath\'eodory convergence and locally uniform convergence}
\label{sec:rem2}

In the classical context, the Carath\'eodory convergence of domains (Definition~\ref{def:kernel}) is associated with the locally uniform convergence of the Riemann maps by the following theorem (see for example Rosenblum and Rovnyak~\cite[p.170]{RR94}):

\begin{theorem}[Carath\'eodory kernel theorem]
\label{thm:kernel}
Let $f_n$ be a conformal mapping from the unit disk $\mathbb{D}$ onto $D_n$ with $f_n(0) = 0$ and $f_n^\prime(0)>0$ for each $n \in \NN$.
Then the following are equivalent:
\begin{enumerate}
\item $(f_n)_{n \in \NN}$ converges to a non-constant function locally uniformly in $\mathbb{D}$;
\item $(D_n)_{n \in \NN}$ converges to a proper subdomain of $\CC$ in Carath\'eodory's sense with respect to the origin.
\end{enumerate}
\end{theorem}

Compared to Theorem~\ref{thm:kernel}, it is naturally expected that, for $\HH$-hulls $(F_n)_{n \in \NN}$ and the corresponding conformal mappings $(f_{F_n})_{n \in \NN}$ with \eqref{eq:hyd} and \eqref{eq:res}, the following are equivalent:
\begin{enumerate}
\renewcommand{\theenumi}{\alph{enumi}}
\renewcommand{\labelenumi}{$\text{\theenumi}^\prime$)}
\item \label{item:funct_conv2}
$(f_{F_n})_{n \in \NN}$ converges locally uniformly in $\HH$;
\item \label{item:hull_conv2}
$(F_n)_{n \in \NN}$ converges in the sense of Definition~\ref{def:hull_continuity}.
\end{enumerate}
However, this equivalence is not obvious from the classical Carath\'eodory kernel theorem.
The point is that, whereas the Riemann maps in Theorem~\ref{thm:kernel} fix the origin, an interior point of $\mathbb{D}$, the mappings $f_{F_n}$ fix the point at infinity, a boundary point of $\HH$.

Although we omit their details,
there are several variants of Carath\'eodory's kernel theorem known.
We here notice that, in the case that the $\HH$-hulls $F_n$ are uniformly bounded, the author gave a proof~\cite[Theorem~3.8]{Mu19spa} of the equivalence of $\text{\ref{item:funct_conv2}}^\prime$) and $\text{\ref{item:hull_conv2}}^\prime$).

\begin{remark} \label{rem:jw2}
We have seen that, at this moment, the probabilistic definition~\eqref{eq:def} of half-plane capacity is not completely the same as the analytic one~\eqref{eq:res} for unbounded $\HH$-hulls.
In the joint work~\cite{HHMS21+} referred to in Remark~\ref{rem:jw},
it is proved in an analytic way that, under assumptions like \eqref{eq:hyd} and \eqref{eq:res}, the locally uniform convergence $f_{F_n} \to f_{F_\infty}$ implies $a_{F_n} \to a_{F_\infty}$.
This statement is thus different than Theorem~\ref{thm:main1} from a technical point of view.
As it exceeds the scope of this article, we just point out this difference and do not go into the detail of such a technicality here.
\end{remark}

\appendix
\section{A result on hitting probability}
\label{sec:appendix}

Let $Z = ((Z_t)_{t \geq 0}, (\PP_z)_{z \in \CC})$
be a complex Brownian motion.
We recall and prove \eqref{eq:radius0}:

\begin{proposition}
\label{prop:A1}
Let $S$ be a closed set in $\overline{\HH}$ and $K$ be a compact set in $\HH$.
Suppose $S \cap K = \emptyset$ and $\Im S < \infty$.
Then
\begin{equation}
\label{eq:A1}
\lim_{\varepsilon \to 0} \sup \left\{\, \PP_w\left(\sigma_{B(z, \varepsilon)} < \tau_{\overline{\HH}}\right) \mathrel{;} z \in S,\ w \in K \,\right\} = 0.
\end{equation}
Here, $\sigma_{\cdot}$ and $\tau_{\cdot}$ denote the first hitting and exit times of $Z$, respectively. 
\end{proposition}

\begin{proof}
We reduce the problem to the case in which $S$ is compact.
Let $a>0$ be such that
$K \subset \{\, z \mathrel{;} \lvert \Re z \rvert < a \,\}$.
We take $\varepsilon_0>0$ so small that a compact set
\[
S^\prime := (S \cap \{\, z \mathrel{;} \lvert \Re z \rvert < a \,\}) \cup \{\, z \mathrel{;} a \leq \lvert \Re z \rvert \leq 2a + \varepsilon_0,\ 0 \leq \Im z \leq \Im S \,\}.
\]
satisfies $\dist(K, S^\prime) \ge \varepsilon_0$.
It follows that
\begin{equation}
\label{eq:A2}
\sup_{\substack{z \in S \\ w \in K}} \PP_w\left(\sigma_{B(z, \varepsilon)} < \tau_{\overline{\HH}}\right)
\leq \sup_{\substack{z \in S^\prime \\ w \in K}} \PP_w\left(\sigma_{B(z, \varepsilon)} < \tau_{\overline{\HH}}\right)
\quad \text{for}\ \varepsilon \in (0, \varepsilon_0).
\end{equation}
This inequality is seen as follows:
Let $L := \{\, z \mathrel{;} \Re z = 2a + \varepsilon_0 \,\}$.
For $z \in S$ with $2a + \varepsilon_0 < \Re z \leq 3a + \varepsilon_0$, we have
\begin{align*}
\PP_w\left(\sigma_{B(z, \varepsilon)} < \tau_{\overline{\HH}}\right)
&= \EE_w\left[ \PP_{Z_{\sigma_L}}\left(\sigma_{B(z, \varepsilon)} < \tau_{\overline{\HH}}\right) ; \sigma_L < \tau_{\overline{\HH}} \right] \\
&= \EE_w\left[ \PP_{Z_{\sigma_L}}\left(\sigma_{B(2a + \varepsilon_0 - \bar{z}, \varepsilon)} < \tau_{\overline{\HH}}\right) ; \sigma_L < \tau_{\overline{\HH}} \right] \\
&= \PP_w\left(\sigma_L < \sigma_{B(2a + \varepsilon_0 - \bar{z}, \varepsilon)} < \tau_{\overline{\HH}}\right)
\leq \PP_w\left(\sigma_{B(2a + \varepsilon_0 - \bar{z}, \varepsilon)} < \tau_{\overline{\HH}}\right).
\end{align*}
Hence
\[
\sup_{\substack{z \in S,\ 2a + \varepsilon_0 < \Re z \leq 3a + \varepsilon_0 \\ w \in K}} \PP_w\left(\sigma_{B(z, \varepsilon)} < \tau_{\overline{\HH}}\right)
\leq \sup_{\substack{z \in S^\prime \\ w \in K}} \PP_w\left(\sigma_{B(z, \varepsilon)} < \tau_{\overline{\HH}}\right).
\]
Repeating such a reflection argument, we can conclude \eqref{eq:A2}.

By virtue of \eqref{eq:A2}, it suffices to prove \eqref{eq:A1} with $S$ replaced by $S^\prime$.
We define half-planes
$\overline{\HH} - z := \{\, w - z \mathrel{;} w \in \overline{\HH} \,\}$, $z \in \CC$,
and $H := \overline{\HH} - i\Im S$.
Then
\[
\PP_w\left(\sigma_{B(z, \varepsilon)} < \tau_{\overline{\HH}}\right)
= \PP_{w-z}\left(\sigma_{B(0, \varepsilon)} < \tau_{\overline{\HH}-z}\right)
\leq \PP_{w-z}\left(\sigma_{B(0, \varepsilon)} < \tau_H\right).
\]
We now put $K - S^\prime := \{\, w - z \mathrel{;} w \in K,\ z \in S^\prime \,\}$
and $f_{\varepsilon}(\zeta) := \PP_\zeta\left(\sigma_{B(0, \varepsilon)} < \tau_{H}\right)$ for $\zeta \in K-S^\prime$.
The function $f_\varepsilon(\zeta)$ is continuous in $\zeta$ and decreases to zero as $\varepsilon \to 0$ for each $\zeta$.
By Dini's theorem, we have
$\lim_{\varepsilon \to 0} \sup_{\zeta \in K-S^\prime} f_{\varepsilon}(\zeta) = 0$,
which gives \eqref{eq:A1}.
\end{proof}

\section{Weak convergence of harmonic measures}
\label{sec:bry}

Following Binder, Rojas and Yampolsky~\cite[(3.1)]{BRY19},
we define the uniform regularity of domains as follows:

\begin{definition}[Uniform regularity]
\label{def:ur}
A collection $\mathcal{D}$ of proper subdomains of $\CC$ is said to be \emph{uniformly regular} if, for any $\varepsilon>0$ there exists $\delta \in (0, \varepsilon)$ such that every $D \in \mathcal{D}$ satisfies
\begin{equation}
\label{eq:ur}
\Hm_D(z, B(z, \varepsilon)) > 1 - \varepsilon
\quad \text{for any}\ z \in D\ \text{with}\ \dist(z, \partial D) < \delta.
\end{equation}
\end{definition}

Only is the Euclidean distance used in Definition~\ref{def:ur}.
This definition is slightly different from the original one, which involves the spherical distance.
In the subsequent argument, we use only the Euclidean distance and do not consider the spherical one.

The aim of this appendix is to provide a self-contained proof of the following theorem, which is part of Binder, Rojas and Yampolsky~\cite[Theorems~2.3 and 3.1]{BRY19}:

\begin{theorem}
\label{thm:BRY19}
Let $D_n$, $n \in \NN$, and $D_\infty$ be proper subdomains of $\CC$ which have a point $z_0$ in common.
Suppose that these domains are uniformly regular.
If $D_n$ converges to $D_\infty$ as $n \to \infty$ in Carath\'eodory's sense with respect to $z_0$,
then $\Hm_{D_n}(z_0, {\cdot})$ converges weakly to $\Hm_{D_\infty}(z_0, {\cdot})$.
\end{theorem}

\begin{remark}
\label{rem:SW08}
The weak convergence of harmonic measures (in simply connected domains) was also proved under a different (but closely related) assumption in the earlier paper of Snipes and Ward~\cite[Remark~1]{SW08} in the context of harmonic measure distribution functions.
\end{remark}

In what follows, we keep the assumption of Theorem~\ref{thm:BRY19}.
Namely, $D_n$, $n \in \NN$, and $D_\infty$ are uniformly regular domains, and $D_n$ converges to $D_\infty$ in Carath\'eodory's sense with respect to a point $z_0$.

We construct a ``common interior approximation'' of $D_n$ for sufficiently large $n$ (cf.\ \cite[Definition~2.2]{BRY19}).
Let $\delta_0 := \dist(z_0, \partial D_\infty) > 0$.
For $\delta \in (0, \delta_0)$ and $r>0$, we define $U_{\delta, r}$ as the connected component of an open set
\[
\left\{\, z \in D_\infty \mathrel{;} \dist(z, \partial D_\infty) > \frac{\delta}{4},\ \lvert z - z_0 \rvert < r \,\right\}
\]
that contains $z_0$.
We also set
\[
\Gamma_\delta := \left\{\, z \in \CC \mathrel{;} \dist(z, \partial D_\infty) = \frac{\delta}{4} \,\right\}.
\]

\begin{lemma}
\label{lem:int_approx}
For every $\delta \in (0, \delta_0)$ and $r>0$, there exists $N \in \NN$ such that, for any $N < n \le \infty$, the domain $D_n$ enjoys
\begin{equation}
\label{eq:int_approx}
U_{\delta, r} \subset D_n \quad \text{and} \quad
\dist(z, \partial D_n) < \delta \ \text{for all}\ z \in \Gamma_\delta \cap \partial U_{\delta, r}.
\end{equation}
\end{lemma}

\begin{proof}
By definition, we have
\begin{equation}
\label{eq:int_approx_infty}
\overline{U_{\delta, r}} \Subset D_\infty
\quad \text{and} \quad
\dist(z, \partial D_\infty) = \frac{\delta}{4}
\ \text{for all}\ z \in \Gamma_\delta \cap \partial U_{\delta, r}.
\end{equation}
Since $D_n \to D_\infty$, there exists $N^\prime \in \NN$ such that $\overline{U_{\delta, r}} \subset \bigcap_{N^\prime \le n < \infty}D_n$.
We here see that, for a fixed $z \in \Gamma_\delta \cap \partial U_{\delta, r}$, there are only finitely many $n$ such that $\dist(z, \partial D_n) \ge \delta/2$;
Otherwise, $U_{\delta, r} \cup B(z, \delta/2)$ would be a subdomain of $D_n$ for infinitely many $n$,
and hence this domain would be a subset of $D_\infty$ by Definition~\ref{def:kernel}.
This contradicts \eqref{eq:int_approx_infty}.
In this way, we can define
\[
n(z) := \max \{\, n \ge N^\prime \mathrel{;} \dist(z, \partial D_n) \ge \delta/2 \,\}
\]
with the maximum set to be $N^\prime$ if this set is empty.

By the compactness, we can choose finitely many points $z_k \in \Gamma_\delta \cap \partial U_{\delta, r}$ so that $\Gamma_\delta \cap \partial U_{\delta, r} \subset \bigcup_k B(z_k, \delta/2)$.
Set $N := \max_k n(z_k)$.
For any $z \in \Gamma_\delta \cap \partial U_{\delta, r}$,
we can find $z_k$ with $z \in B(z_k, \delta/2)$, and hence
\[
\dist(z, \partial D_n)
\le \lvert z - z_k \rvert + \dist(z_k, \partial D_n)
<\frac{\delta}{2} + \frac{\delta}{2} = \delta
\]
for all $n > N$.
\end{proof}

Let $Z = ((Z_t)_{t \ge 0}, (\PP_z)_{z \in \CC})$ be a planar Brownian motion.
We write the exit time of $Z$ from $U_{\delta, r}$ as $\tau_{\delta, r} := \tau_{U_{\delta, r}}$.

\begin{lemma}
\label{lem:recurrence}
For every $\varepsilon>0$, there exists $R>0$ such that
\begin{equation}
\label{eq:recurrence}
\PP_{z_0}\left( Z_{\tau_{\delta, r}} \in \Gamma_\delta \right) \ge 1 -\varepsilon
\end{equation}
for all $\delta \in (0, \delta_0)$ and $r>R$.
\end{lemma}

\begin{proof}
Since $D_\infty$ is a regular domain by assumption, its complement is non-polar.
Hence we can take a compact non-polar subset $K$ of $D_\infty^c$ (see for instance Port and Stone~\cite[Proposition~2.4, Chapter~2]{PS78}).
In two dimension, a non-polar set is recurrent for Brownian motion \cite[Propositions~2.9 and 2.10, Chapter~2]{PS78}, which means
\[
\lim_{r \to \infty}\PP_{z_0}\left(\sigma_K < \sigma_{\partial B(z_0, r)}\right) = \PP_{z_0}\left(\sigma_K < \infty\right) = 1.
\]
Here, the Brownian motion starting at $z_0$ has to hit $\Gamma_\delta$ before it hits $K \subset D_\infty^c$.
Therefore, for any $\varepsilon>0$, there exists $R>0$ such that, if $r>R$, then
\[
1 - \varepsilon
\le \PP_{z_0}\left(\sigma_K < \sigma_{\partial B(z_0, r)}\right)
\le \PP_{z_0}\left( Z_{\tau_{\delta, r}} \in \Gamma_\delta \right).
\qedhere
\]
\end{proof}

Let $\Lip_b(\CC)$ be the space of bounded Lipschitz functions in $\CC$ with norm
\[
\lVert f \rVert_{\Lip_b(\CC)}
:= \sup_{z \in \CC} \lvert f(z) \rvert + \sup_{\substack{z,w \in \CC \\ z \neq w}}\frac{\lvert f(z)-f(w) \rvert}{\lvert z-w \rvert}.
\]
A distance
\[
d(\mu, \nu)
:= \sup_{\substack{f \in \Lip_b(\CC) \\ \lVert f \rVert_{\Lip_b(\CC)} \le 1}}\left\lvert \int_{\CC} f \, d\mu - \int_{\CC} f \, d\nu \right\rvert,
\quad \mu, \nu \in \mathcal{P}(\CC),
\]
is known to metrize the weak topology of the space $\mathcal{P}(\CC)$ of Borel probability measures on $\CC$ (see for instance Dudley~\cite[Theorem~11.3.3]{Du02}).
Using this distance, we now prove Theorem~\ref{thm:BRY19}.

\begin{proof}[Proof of Theorem~\ref{thm:BRY19}]
Fix $\varepsilon>0$.
By Lemma~\ref{lem:recurrence} and the assumption of the uniform regularity,
we can take $\delta \in (0, \delta_0)$ and $r>0$ so that \eqref{eq:recurrence} and \eqref{eq:ur} with $D=D_n$ for all $n \in \NN \cup \{\infty\}$ hold.
For these $\delta$ and $r$, there exists $N \in \NN$ such that \eqref{eq:int_approx} holds for all $N < n \le \infty$ by Lemma~\ref{lem:int_approx}.

For $N < n \le \infty$ and $f \in \Lip_b(\CC)$ with $\lVert f \rVert_{\Lip_b(\CC)} \leq 1$,
let us consider the difference
\[
\left\lvert \int_{\CC} f(z) \Hm_{U_{\delta,r}}(z_0, dz) - \int_{\CC} f(z) \Hm_{D_n}(z_0, dz) \right\rvert
\le \EE_{z_0}\left[\lvert f(Z_{\tau_{\delta,r}}) - f(Z_{\tau_{D_n}})\rvert\right].
\]
Note that the random variable $\Delta_n := \lvert f(Z_{\tau_{\delta,r}}) - f(Z_{\tau_{D_n}}) \rvert$ satisfies
\[
\Delta_n \le \lVert f \rVert_{\Lip_b(\CC)}\lvert Z_{\tau_{\delta,r}}-Z_{\tau_{D_n}}\rvert \le \lvert Z_{\tau_{\delta,r}}-Z_{\tau_{D_n}}\rvert
\]
and
\[
\Delta_n \le 2\lVert f \rVert_{\Lip_b(\CC)} \le 2.
\]
We take the expectation of $\Delta_n$ on three disjoint events.
First,
\begin{equation}
\label{eq:EE_W1}
\EE_{z_0}\left[\Delta_n; Z_{\tau_{\delta,r}} \in \Gamma_\delta,\ \lvert Z_{\tau_{\delta,r}}-Z_{\tau_{D_n}}\rvert < \varepsilon\right]
\le \varepsilon.
\end{equation}
Second,
\begin{align}
&\EE_{z_0}\left[\Delta_n; Z_{\tau_{\delta,r}} \in \Gamma_\delta,\ \lvert Z_{\tau_{\delta,r}}-Z_{\tau_{D_n}}\rvert \ge \varepsilon\right] \notag \\
&\le 2 \PP_{z_0}\left(Z_{\tau_{\delta,r}} \in \Gamma_\delta,\ \lvert Z_{\tau_{\delta,r}}-Z_{\tau_{D_n}}\rvert \ge \varepsilon\right) \notag \\
&= 2 \EE_{z_0}\left[\Hm_{D_n}(Z_{\tau_{\delta,r}}, B(Z_{\tau_{\delta,r}}, \varepsilon)^c); Z_{\tau_{\delta,r}} \in \Gamma_\delta\right]
\le 2 \varepsilon. \label{eq:EE_W2}
\end{align}
Third,
\begin{equation}
\label{eq:EE_W3}
\EE_{z_0}\left[\Delta_n; Z_{\tau_{\delta,r}} \notin \Gamma_{\delta}\right]
\leq 2 \PP_{z_0}(Z_{\tau_{\delta,r}} \notin \Gamma_{\delta})
\le 2\varepsilon.
\end{equation}
Combining \eqref{eq:EE_W1}--\eqref{eq:EE_W3} yields $\EE_{z_0}[\Delta_n] \le 5\varepsilon$.
Since $f$ was arbitrary, we have
\[
d(\Hm_{U_{\delta,r}}(z_0, {\cdot}), \Hm_{D_n}(z_0, {\cdot})) \le 5\varepsilon
\]
for all $N < n \le \infty$.
Finally, it follows from the triangle inequality that
\[
d(\Hm_{D_n}(z_0, {\cdot}), \Hm_{D_\infty}(z_0, {\cdot})) \le 10\varepsilon
\]
for all $n > N$, which proves Theorem~\ref{thm:BRY19}.
\end{proof}


\end{document}